\documentclass[11pt]{article}
\usepackage{amsmath, amssymb, amsthm}
\usepackage{verbatim}
\usepackage{multicol}
\usepackage{enumerate}
\usepackage{comment}
\usepackage[none]{hyphenat}
\usepackage{hyperref}
\hypersetup{
	colorlinks=true,
	linkcolor=blue,
	filecolor=magenta,
	urlcolor=cyan,
	citecolor=blue
}

\usepackage{pgf}
\usepackage{tikz}
\usepackage{ifthen}
\usetikzlibrary{math}
\usetikzlibrary{positioning,arrows,shapes,decorations.markings,decorations.pathreplacing,matrix,patterns}
\tikzstyle{vertex}=[circle,draw=black,fill=black,inner sep=0,minimum size=3pt,text=white,font=\footnotesize]

\date{}
\title{\vspace{-0.8cm}A sharp threshold phenomenon in string graphs}
\author{
	Istv\'{a}n Tomon \thanks{\'{E}cole Polytechnique F\'{e}d\'{e}rale de Lausanne, Research partially supported by Swiss National Science Foundation grants no. 200020-162884 and 200021-175977. \emph{e-mail}: \textbf{istvantomon@gmail.com}} 
}

\theoremstyle{plain}
\newtheorem{theorem}{Theorem}

\newtheorem{claim}[theorem]{Claim}
\newtheorem{lemma}[theorem]{Lemma}
\newtheorem{conjecture}[theorem]{Conjecture}

\newtheorem{prop}[theorem]{Proposition}
\newtheorem{obs}{Observation}
\newtheorem*{reglemma}{Regularity Lemma}

\theoremstyle{definition}

\begin{document}

\maketitle
\sloppy

\begin{abstract}
%A \emph{string graph} is the intersection graph of a family of curves in the plane. It was proved by Lee (2017) that if $G$ is a string graph with $m$ edges, then $G$ contains a separator of size $O(\sqrt{m})$. An immediate consequence of this result is that there exists some absolute constant $c>0$ such that if $G$ is a string graph with $n$ vertices and at most $cn^{2}$ edges, then $V(G)$ contains two disjoint linear sized subsets $A$ and $B$ such that there are no edges between $A$ and $B$. We determine the exact threshold for which these linear sized subsets start to appear. 

%We prove that for every $\epsilon>0$ there exists $\delta>0$ such that if $G$ is a string graph with $n$ vertices and at most $(\frac{1}{4}-\epsilon)\frac{n^{2}}{2}$ edges, then $V(G)$ contains two disjoint subsets $A$ and $B$ of size $|A|=|B|\geq \delta n$ such that there are no edges between $A$ and $B$. On the other hand, for every positive integer $n$ there exists a string graph $G$ with $n$ vertices and less than $(\frac{1}{4}+\epsilon)\frac{n^{2}}{2}$ edges such that for every $A,B\subset V(G)$ if $|A|=|B|$ and there are no edges between $A$ and $B$, then $|A|=|B|=O(\frac{1}{\epsilon}\log n)$.

We prove that for every $\epsilon>0$ there exists $\delta>0$ such that the following holds. Let $\mathcal{C}$ be a collection of $n$ curves in the plane such that there are at most $(\frac{1}{4}-\epsilon)\frac{n^{2}}{2}$ pairs of curves  $\{\alpha,\beta\}$ in $\mathcal{C}$ having a nonempty intersection.  Then $\mathcal{C}$ contains two disjoint subsets $\mathcal{A}$ and $\mathcal{B}$ such that $|\mathcal{A}|=|\mathcal{B}|\geq \delta n$, and every $\alpha\in \mathcal{A}$ is disjoint from every $\beta\in\mathcal{B}$. %This verifies a recent conjecture of Pach and Tomon, who proved the same statement for $x$-monotone curves.

On the other hand, for every positive integer $n$ there exists a collection $\mathcal{C}$ of $n$ curves in the plane such that there at most $(\frac{1}{4}+\epsilon)\frac{n^{2}}{2}$ pairs of curves  $\{\alpha,\beta\}$ having a nonempty intersection, but if $\mathcal{A},\mathcal{B}\subset \mathcal{C}$ are such that $|\mathcal{A}|=|\mathcal{B}|$ and $\alpha\cap \beta=\emptyset$ for every $(\alpha,\beta)\in \mathcal{A}\times\mathcal{B}$, then $|\mathcal{A}|=|\mathcal{B}|=O(\frac{1}{\epsilon}\log n)$.
\end{abstract}

\section{Introduction}

The \emph{intersection graph} of family of sets $\mathcal{C}$ is the graph whose vertices are identified with the elements of $\mathcal{C}$, and two vertices are joined by an edge if the corresponding elements of $\mathcal{C}$ have a nonempty intersection. A \emph{curve} in the plane is the image of a continuous function $\phi:[0,1]\rightarrow\mathbb{R}^{2}$, and a \emph{string graph} is the intersection graph of a family of curves.

Combinatorial and computational properties of string graphs are extensively studied, both from a theoretical and a practical point of view. The concept of string graphs was introduced by Benzer \cite{B59} in order to study topological properties of genetic structures. Later, Sinden \cite{S66} considered string graphs to model printed circuits, and he proved the already non-trivial statement that not every graph is a string graph.

A \emph{separator} in a graph $G$ is a subset $S$ of the vertices of $G$ such that every connected component of $G-S$ has size at most $\frac{2|V(G)|}{3}$. A classical result of Lipton and Tarjan \cite{LT79} is that every planar graph on $n$ vertices contains a separator of size $O(\sqrt{n})$. Building on this result, Fox and Pach \cite{FP10_sep1} proved that if $\mathcal{C}$ is a family of curves and $m$ is the total number of crossings between the elements of $\mathcal{C}$, then the intersection graph of $G$ contains a separator of size $O(\sqrt{m})$. They conjectured the strengthening of their result that the same conclusion holds if $m$ denotes the number of intersecting pairs of curves in $\mathcal{C}$. Almost settling this conjecture, Matou\v{s}ek \cite{M14} proved that every string graph $G$ with $m$ edges contains a separator of size $O(\sqrt{m}\log m)$. Recently, the conjecture of Fox and Pach was confirmed by Lee \cite{L17}.   

\begin{theorem}\label{thm:separator}(\cite{L17})
	If $G$ is a string graph with $m$ edges, then $G$ contains a separator of size $O(\sqrt{m})$.
\end{theorem}

 Fox and Pach \cite{FP10_sep1,FP14_sep2} gave several applications of the existence of small separators in string graphs. 
 
 A \emph{bi-clique} in a graph $G$ is a pair of disjoint subsets of vertices $(A,B)$ such that $|A|=|B|$ and $ab\in E(G)$ for every $a\in A$ and $b\in B$, and a \emph{size} of a bi-clique $(A,B)$ is $|A|$. An immediate consequence of Theorem \ref{thm:separator} is that for every $0<\delta<\frac{1}{4}$ there exists $c>0$ such that if $G$ is a string graph with $n$ vertices and at most $cn^{2}$ edges, then the complement of $G$ contains a bi-clique of size at least $\delta n$. Pach and Tomon \cite{PT19} proved that if restrict our attention to \emph{$x$-monotone curves} (curves such that every vertical line intersects them in at most one point), then there is a sharp threshold for the edge density when linear sized bi-cliques start to appear in the complement of $G$. More precisely, they proved that for every $\epsilon>0$ there exists $\delta>0$ such that if $G$ is the intersection graph of $n$ $x$-monotone curves and $|E(G)|\leq (\frac{1}{4}-\epsilon)\frac{n^{2}}{2}$, then $\overline{G}$ contains a bi-clique of size at least $\delta n$. On the other hand, they showed that there exists a family of $n$ convex sets\footnote{Convex sets can be approximated arbitrarily closely by $x$-monotone curves, so intersection graphs of convex sets are also intersection graphs of $x$-monotone curves.} whose intersection graph $G$ has at most $(\frac{1}{4}+\epsilon)\frac{n^{2}}{2}$ edges, but the size of the largest bi-clique in $\overline{G}$ is $O(\frac{1}{\epsilon}\log n)$. Indeed, the existence of such families follows from the following result of Pach and T\'oth \cite{PT06+}: if $G$ is a graph whose vertex set can be partitioned into four parts $V_{1},V_{2},V_{3},V_{4}$ such that $V_{i}$ spans a clique for $i=1,2,3,4$, then $G$ can be realized as the intersection graph of convex sets. See Figure \ref{figure6} for a brief explanation. But then, a standard probabilistic construction shows that there exist such graphs $G$ with $n$ vertices, at most $(\frac{1}{4}+\epsilon)\frac{n^{2}}{2}$ edges such that the size of the largest bi-clique in $\overline{G}$ is $O(\frac{1}{\epsilon}\log n)$.    
 
 \begin{figure}[t]
 	\begin{center}
 	\begin{tikzpicture}[scale=1]
 	
 	\draw[dotted] (-1.2,0.5) -- (-5,2) ;
 	\draw[dotted] (-1.2,-0.5) -- (-5,0) ;
 	\draw[dotted] (-2.52,0.5) -- (-7.64,2) ;
 	 \draw[dotted] (-2.52,-0.5) -- (-7.64,0) ;
 	\draw[dashed] (-1.2,0.5) rectangle (-2.52,-0.5)  ;
  	\draw[dashed, fill=white] (-5,2) rectangle (-7.64,0) ;
  	
  	\draw (-7.64,1.1) -- (-5,1.6) ;
  	\draw (-7.64,1.6) -- (-5,1.1) ;
  	\draw (-7.64,1.28) -- (-5,1.28) ;
  	
  	\draw[red] (-7.64,0.9) -- (-7.4,1.2) -- (-5.24,1.2) -- (-5,0.9) ;
  	\draw[red] (-7.64,0.7) -- (-6.32,1.32) -- (-5,0.7) ;
 	
   \draw (0,0) circle (0.5) ;
  % \draw (0:3.73) circle (3.23) ;
  % \draw (120:3.73) circle (3.23) ;
  % \draw (240:3.73) circle (3.23) ;
   
    \draw [domain=120:240] plot ({3.73+3.23*cos(\x)}, {3.23*sin(\x)});
    \draw [domain=240:360] plot ({-1.865+3.23*cos(\x)}, {3.23+3.23*sin(\x)});
    \draw [red,domain=0:120] plot ({-1.865+3.23*cos(\x)}, {-3.23+3.23*sin(\x)});
    
    \node[] at (0,0) {$C_{1}$} ;
    \node[] at (0:3) {$C_{2}$} ;
    \node[] at (120:3) {$C_{3}$} ;
    \node[] at (240:3) {$C_{4}$} ;
    \node[vertex] at (0.5,0) {}; \node[] at (0.9,0.1) {$p_{12}$} ;
    \node[vertex] at (120:0.5) {}; \node[] at (120:0.8) {$p_{13}$} ;
	\node[vertex] at (240:0.5) {}; \node[] at (240:0.8) {$p_{14}$} ;
	
	\node[vertex] at (60:1.86) {}; \node[] at (47:1.86) {$p_{23}$} ;
	\node[vertex] at (180:1.86) {}; \node[] at (172:1.86) {$p_{34}$} ;
	\node[vertex] at (300:1.86) {}; \node[] at (287:1.86) {$p_{24}$} ;

 \end{tikzpicture}
 		\caption{Consider four pairwise touching circles $C_{1},C_{2},C_{3},C_{4}$ that touch at the six points $p_{ij}$ for $1\leq i<j\leq 4$. For each vertex in $V_{i}$, we can define a convex set that only slightly deviates from the circle $C_{i}$. It is possible to define these convex sets such that if $x\in V_{i}$ and $y\in V_{j}$ for some $1\leq i<j\leq 4$, then the corresponding convex sets intersect in the small neighborhood of $p_{ij}$ if $xy\in E(G)$, and these convex sets are disjoint otherwise.}
 \label{figure6}
\end{center}
\end{figure}
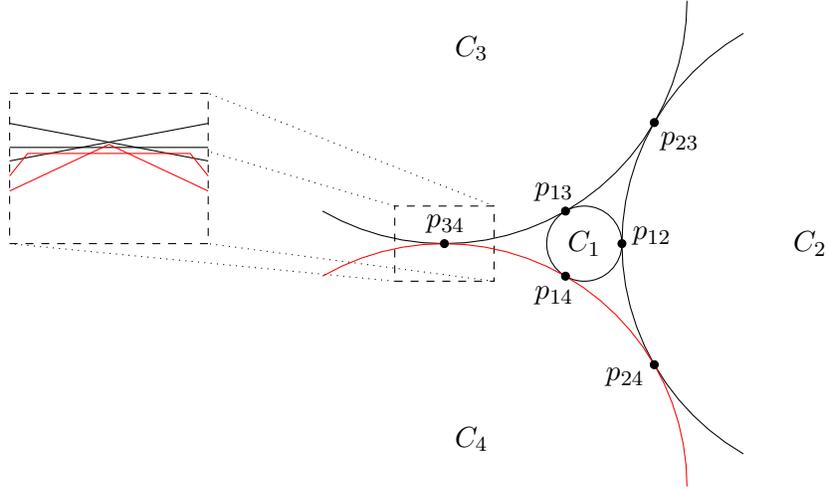
 
 Pach and Tomon \cite{PT19} also conjectured that their result holds without the assumption that the curves are $x$-monotone. Our main result is the proof of this conjecture.

\begin{theorem}\label{thm:mainthm}
	For every $\epsilon>0$ there exists $\delta>0$ such that the following holds. For every positive integer $n$ if $G$ is a string graph with $n$ vertices and at most $(\frac{1}{4}-\epsilon)\frac{n^{2}}{2}$ edges, then there exist two disjoint sets $A,B\subset V(G)$ such that $|A|=|B|\geq \delta n$ and there are no edges between $A$ and $B$. 
\end{theorem}   

The problem of finding large bi-cliques in intersection graphs and in their complements gained a lot of interest recently. Fox, Pach and T\'oth \cite{FPT10} proved that if $G$ is the intersection graph of $n$ convex sets, then either $G$ or its complement contains a bi-clique of size $\Omega(n)$. Also, if $G$ is the intersection graph  of $n$ $x$-monotone curves, then either $G$ contains a bi-clique of size $\Omega(\frac{n}{\log n})$, or $\overline{G}$ contains a bi-clique of size $\Omega(n)$, and these bounds are best possible. Pach and Tomon \cite{PT19} gave a different proof of the latter result using ordered graphs. Fox, Pach and T\'oth \cite{FPT11} also proved that for every $k\in\mathbb{N}$ there exists $c_{k}>0$ such that if $G$ is the intersection graph of $n$ curves, where any two of the curves cross at most $k$ times, then either $G$ or its complement contains a bi-clique of size at least $c_{k}n$. 

In the case there are no restriction on the curves, Fox and Pach \cite{FP12_inc} proved that every dense string graph contains a dense incomparability graph as subgraph. By results of Fox \cite{F06} and Fox, Pach and T\'oth \cite{FPT10}, this implies that every dense string graph contains a bi-clique of size $\Omega(\frac{n}{\log n})$. Combining this with the result of Lee \cite{L17}, we get that if $G$ is a string graph, either $G$ contains a bi-clique of size $\Omega(\frac{n}{\log n})$, or $\overline{G}$ contains a bi-clique of size $\Omega(n)$; again, these bounds are best possible. Let us remark that incomparability graphs exhibit a similar threshold phenomenon as Theorem \ref{thm:mainthm}. Indeed, in case $G$ is an incomparability graph with $n$ vertices and at most $(\frac{1}{2}-\epsilon)\frac{n^{2}}{2}$ edges, then $\overline{G}$ contains a bi-clique of size $\Omega(\epsilon n)$, but there are incomparability graphs $G$ with at most $(\frac{1}{2}+\epsilon)\frac{n^{2}}{2}$ edges such that the size of the largest bi-clique in $\overline{G}$ is $O(\frac{1}{\epsilon}\log n)$, see \cite{FPT10}.

Finally, the existence of linear sized bi-cliques in the complement of not too dense string graphs also follows from a recent graph theoretic result of Chudnovsky, Scott, Seymour and Spirkl \cite{CSSS18}. 

\begin{theorem}(\cite{CSSS18})\label{thm:subdivision}
	Let $H$ be a graph. Then there exists $\epsilon>0$ such that if $G$ is a graph with $n$ vertices and at most $\epsilon n^{2}$ edges, and $G$ does not contain a subdivision of $H$ as an induced subgraph, then $\overline{G}$ contains a bi-clique of size at least $\epsilon n$.
\end{theorem}

 Indeed, if $G$ is a string graph, then $G$ does not contain a subdivision of $K_{5}$ as an induced subgraph (we discuss this in more detail below). In order to prove Theorem \ref{thm:mainthm}, we shall also consider graphs which avoid certain weaker notions of subdivisions of $K_{5}$ as induced subgraph.

Our paper is organized as follows. In Section \ref{sect:proof}, we prove Theorem \ref{thm:mainthm}. In Section \ref{sect:remarks}, we conclude our paper with some open problems and remarks. But first, let us agree on some terminology.

\subsection{Preliminaries and notation}

Given two curves $\alpha$ and $\beta$, a \emph{crossing} between $\alpha$ and $\beta$ is a point $z\in\alpha\cap\beta$ such that $\alpha$ passes to the other side of $\beta$ at $z$. Given a collection of curves $\mathcal{C}$, we assume that if two curves $\alpha,\beta\in\mathcal{C}$ intersect at some point $z$, then $z$ is a crossing between $\alpha$ and $\beta$. Indeed, by a small perturbation any intersection point can be turned into a crossing without changing the intersection graph. Also, we assume that there are a finite number of crossings between any two curves. Indeed, by a result of Schaefer and \v{S}tefankovi\v{c} \cite{SS04}, every string graph can be realized as the intersection graph of such a collection of curves.

For a positive integer $t$, $K_{t}$ denotes the complete graph on $t$ vertices. Let $G$ be a graph. The complement of $G$ is denoted by $\overline{G}$. If $v\in V(G)$, then $N(v)=\{w\in V(G):vw\in E(G)\}$ is the neighborhood of $v$. Also, if $U$ is a subset of the vertices, then $G[U]$ is the subgraph of $G$ induced by $U$. Moreover, if $U$ and $V$ are disjoint subsets of $V(G)$, then $E(U,V)$ is the set of edges in $G$ with one endpoint in $U$ and one endpoint in $V$.

\section{Proof of Theorem \ref{thm:mainthm}}\label{sect:proof}

In this section, we prove Theorem \ref{thm:mainthm}. First, we shall reduce Theorem \ref{thm:mainthm} to a completely graph theoretic statement. In order to do this, we make use of the following immediate consequence of Theorem \ref{thm:separator}.

\begin{lemma}\label{lemma:separator}
	There exists a constant $C>0$ such that for every positive integer $n$, if $G$ is a string graph with $n$ vertices and at most $Cn^{2}$ edges, then $\overline{G}$ contains a bi-clique of size at least $\frac{n}{4}$. 
\end{lemma}

Note that the constant $\frac{1}{4}$ has no significance in this lemma, any positive constant would serve our purposes. Therefore, instead of citing Theorem \ref{thm:separator}, we could have cited Theorem \ref{thm:subdivision} as well to get somewhat weaker version of this lemma. In the rest of the paper, $C$ denotes the constant described by Lemma \ref{lemma:separator}. This lemma tells us that in order to prove Theorem \ref{thm:mainthm}, it is enough to consider dense graphs. But for dense graphs, we can use powerful tools such as the Regularity lemma.

 Let us introduce some of the main notions used in this section. Let $G$ be a graph with $n$ vertices. If $0<\delta<\frac{1}{2}$, say that  $G$  is \emph{$\delta$-full} if for every $A,B\subset V(G)$ satisfying $|A|,|B|\geq \delta n$ there exists an edge between $A$ and $B$. The \emph{density} of $G$ is $d(G)=\frac{|E(G)|}{n^{2}}.$ Say that $G$ is \emph{$(\alpha,\beta)$-dense} if every induced subgraph of $G$ with at least $\alpha n$ vertices has density at least $\beta$.  By Lemma \ref{lemma:separator}, if $G$ is a string graph that is $\delta$-full, then $G$ is $(4\delta,C)$-dense.

If $H$ is a graph, a \emph{$k$-subdivision} of an edge $xy\in E(H)$ is the following operation: we replace the edge $xy$ by a path $x=x_{0},x_{1},\dots,x_{k+1}=y$, where $\{x_{1},\dots,x_{k}\}$ is disjoint from $V(H)$. A subdivision of an edge is a $k$-subdivision for some $k\geq 1$. A graph $H'$ is a $k$-subdivision (or subdivision) of $H$ if we get $H'$ from $H$ by $k$-subdividing (or subdividing) every edge of $H$. Also, $H'$ is a \emph{partial subdivision} of $H$ if we get $H'$ from $H$ by subdividing some (possibly zero) edges of $H$.  If $H'$ is a partial subdivision of $H$, we refer to the vertices of $H$ in $H'$ as \emph{branch-vertices}, and we refer to the other vertices as \emph{side-vertices}.

If $H'$ is a $k$-subdivision (or subdivision) of $H$, let $H''$ be a graph we get from $H'$ by adding edges between some pairs of side-vertices that belong to the subdivisions of neighboring edges of $H$. Then $H''$ is called a \emph{weak-$k$-subdivision} (or weak-subdivision) of $H$. More precisely, $H''$ is a \emph{weak-$k$-subdivision} (or weak-subdivision) of $H$ if there exists a $k$-subdivision (or subdivision) $H'$ of $H$ such that ${V(H'')=V(H')}$, $E(H')\subset E(H'')$ and for every $uv\in E(H'')\setminus E(H')$, $u$ and $v$ are side-vertices of $H'$, and if $u$ belongs to the subdivision of the edge $xy\in E(H)$ and $v$ belongs to the subdivision of $x'y'\in E(H)$, then $xy\neq x'y'$, and $xy$ and $x'y'$  share an endpoint. See Figure \ref{figure1} for an example.

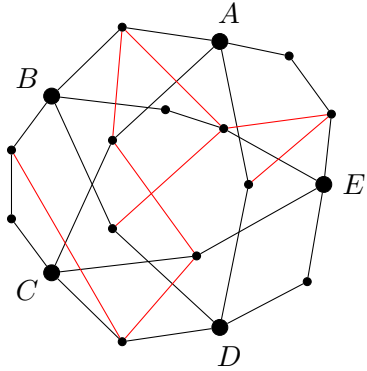
\begin{figure}[t]
	\begin{center}
		\begin{tikzpicture}[scale=1]
		
	  \node[vertex,minimum size=6pt] (A) at (72:2) {} ;
	  \node[vertex,minimum size=6pt] (B) at (144:2) {} ;
	  \node[vertex,minimum size=6pt] (C) at (216:2) {} ;
	  \node[vertex,minimum size=6pt] (D) at (288:2) {} ;
	  \node[vertex,minimum size=6pt] (E) at (0:2) {} ; 
	  
	  \node[] at (72:2.4) {$A$} ;
	  \node[] at (144:2.4) {$B$} ;
	  \node[] at (216:2.4) {$C$} ;
	  \node[] at (288:2.4) {$D$} ;
	  \node[] at (0:2.4) {$E$} ;

	  \node[vertex] (AB) at (108:2.2) {}; \draw (A) -- (AB) -- (B) ;
	  \node[vertex] (BC1) at (168:2.2) {};
	  \node[vertex] (BC2) at (192:2.2) {}; \draw (B) -- (BC1) -- (BC2) -- (C) ;
	  \node[vertex] (CD) at (252:2.2) {}; \draw (C) -- (CD) -- (D) ;
	  \node[vertex] (DE) at (324:2.2) {}; \draw (D) -- (DE) -- (E) ;
	  \node[vertex] (EA1) at (24:2.3) {};
	  \node[vertex] (EA2) at (48:2.3) {}; \draw (E) -- (EA1) -- (EA2) -- (A) ;
	  \node[vertex] (AC) at (144:1) {} ; \draw (A) -- (AC) -- (C) ;	
	  \node[vertex] (BD) at (216:1) {} ; \draw (B) -- (BD) -- (D) ;	
	  \node[vertex] (CE) at (288:1) {} ; \draw (C) -- (CE) -- (E) ;	
	  \node[vertex] (DA) at (0:1) {} ; \draw (D) -- (DA) -- (A) ;	
	  \node[vertex] (EB1) at (48:1) {} ; 
	  \node[vertex] (EB2) at (96:1) {} ; \draw (E) -- (EB1) -- (EB2) -- (B) ;
	  
      \draw[red] (AB) -- (EB1) ; \draw[red] (AB) -- (AC) ;  \draw[red] (CD) -- (CE) ; \draw[red] (CD) -- (BC1) ; 
       \draw[red] (EA1) -- (DA) ; \draw[red] (EA1) -- (EB1) ;  \draw[red] (EB1) -- (BD) ; \draw[red] (AC) -- (CE) ;

		\end{tikzpicture}
		\caption{A weak-subdivision of $K_{5}$, where $A,B,C,D,E$ are the branch-vertices. The edges that are part of the weak-subdivision but not the subdivision are red.}
		\label{figure1}
	\end{center}
\end{figure}

Slightly generalizing Lemma 3.2 in \cite{PT06+} and Lemma 11 in \cite{PRY18}, we show that string graphs do not contain weak-subdivisions of $K_{5}$.

\begin{lemma}\label{lemma:K5}
	If $G$ is a string graph, then $G$ does not contain a weak-subdivision of $K_{5}$ as an induced subgraph.
\end{lemma}

\begin{proof}
	It is enough to show that if $H'$ is a weak-subdivision of $K_{5}$, then $H'$ is not a string graph. Suppose that $H'$ is a string graph and let $\mathcal{C}$ be a collection of curves realizing $H'$. Let $v_{1},\dots,v_{5}$ be the branch-vertices of $H'$ and let $C_{i}\in \mathcal{C}$ be the curve corresponding to $v_{i}$ for $i=1,\dots,5$. Let $x_{i}$ be an arbitrary point of $C_{i}$. For every $1\leq i<j\leq 5$, there exists a path connecting $v_{i}$ and $v_{j}$ in $H'$. Let $X_{i,j}$ be the union of the curves in $\mathcal{C}$ corresponding to the vertices of this path. Then $X_{i,j}$ is a connected set in the plane containing $x_{i}$ and $x_{j}$, so there exists a curve $\gamma_{i,j}\subset X_{i,j}$ with endpoints $x_{i}$ and $x_{j}$. Note that as $H'$ is a weak-subdivision, $X_{i,j}\cap X_{i',j'}\neq \emptyset$ if and only if $\{i,j\}\cap \{i',j'\}\neq \emptyset$. Hence, $\cup_{1\leq i<j\leq 5} \gamma_{i,j}$ is a drawing of $K_{5}$ in which if two edges cross, then they share a vertex. 
	
	However, such a drawing can be turned into a planar drawing of $K_{5}$ by repeating the following operation. Suppose that the curves $\gamma_{i,j}$ and $\gamma_{i,k}$ cross in some point $z$. Let $\gamma_{1}$ be the subcurve of $\gamma_{i,j}$ with endpoints $x_{i}$ and $z$, let $\gamma_{2}$ be the subcurve of $\gamma_{i,j}$ with endpoints $z$ and $x_{j}$,  let $\gamma_{3}$ the subcurve of $\gamma_{i,k}$ with endpoints $x_{i}$ and $z$, and let $\gamma_{4}$ be the subcurve of $\gamma_{i,k}$ with endpoints $z$ and $x_{k}$. Replace $\gamma_{i,j}$ with $\gamma_{i,j}'=\gamma_{1}\cup \gamma_{4}$, and replace $\gamma_{i,k}$ with $\gamma_{i,k}'=\gamma_{2}\cup \gamma_{3}$, and perturb $\gamma_{i,j}'$ and $\gamma_{i,k}'$ slightly in the neighborhood of  $z$ such that $\gamma_{i,j}'$ and $\gamma_{i,k}'$ no longer have an intersection point around $z$. See Figure \ref{figure5} for an illustration of this operation. After such an operation, the number of crossings in the drawing of $K_{5}$ decreases, so in a finite number of steps, we get a planar drawing of $K_{5}$. As $K_{5}$ is not planar, this is a contradiction.  
\end{proof}

\begin{figure}[t]
	\begin{center}
		\begin{tikzpicture}[scale=1]
		
		\draw  plot[smooth, tension=.7] coordinates {(5.5,-2) (5,-0.5) (5.45,0.5) (4.5,1) (3,1.5)};
		\draw  plot[smooth, tension=.7] coordinates {(5.5,-2) (6.5,-1) (6.5,0) (5.55,0.5) (6,1.5) (7,2) (8,2)};
		
		\draw  plot[smooth, tension=.7] coordinates {(-3.5,-2) (-4,-0.5) (-3.5,0.5) (-3,1.5) (-2,2) (-1,2)};
		\draw  plot[smooth, tension=.7] coordinates {(-3.5,-2) (-2.5,-1) (-2.5,0) (-3.5,0.5) (-4.5,1) (-6,1.5)};
		
		\node[] at (-3.5,-2.3) {$x_{i}$} ; \node[vertex, minimum size=6pt] at (-3.5,-2) {} ;
		\node[] at (-1,2.3) {$x_{k}$} ; \node[vertex, minimum size=6pt] at (-1,2) {} ;
		\node[] at (-6,1.8) {$x_{j}$} ; \node[vertex, minimum size=6pt] at (-6,1.5) {} ;
		\node[] at (-3.2,0.5) {$z$} ; \node[vertex, minimum size=3pt] at (-3.5,0.5) {} ;
		\node[] at (-2,-0.75) {$\gamma_{1}$} ; \node[] at (-4.3,-0.75) {$\gamma_{3}$} ;
		\node[] at (-4.5,1.2) {$\gamma_{2}$} ; \node[] at (-2.4,2.1) {$\gamma_{4}$} ;
		
		\node[] at (5.5,-2.3) {$x_{i}$} ; \node[vertex, minimum size=6pt] at (5.5,-2) {} ;
		\node[] at (8,2.3) {$x_{k}$} ; \node[vertex, minimum size=6pt] at (8,2) {} ;
		\node[] at (3,1.8) {$x_{j}$} ; \node[vertex, minimum size=6pt] at (3,1.5) {} ;
		\end{tikzpicture}
		\caption{Uncrossing two neighboring edges.}
		\label{figure5}
	\end{center}
\end{figure}
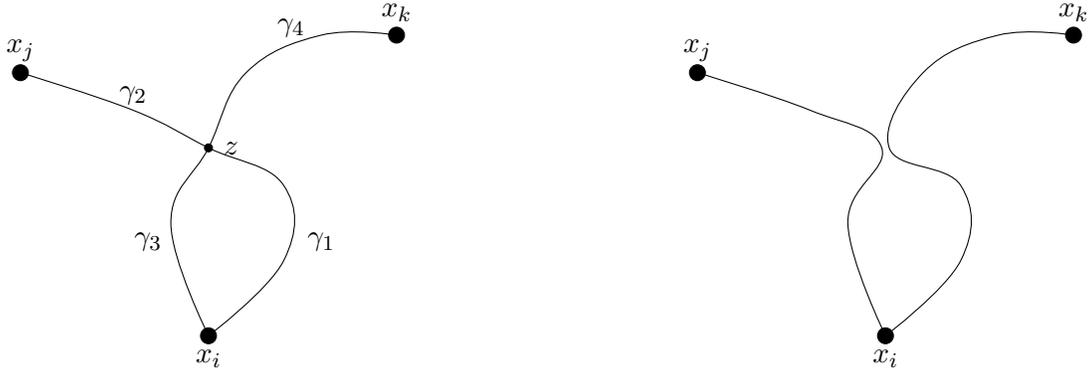

By combining Lemma \ref{lemma:separator} and Lemma \ref{lemma:K5}, it is enough to prove the following graph theoretic statement in order to prove Theorem \ref{thm:mainthm}.

\begin{theorem}\label{thm:weakK5}
 For every $\epsilon>0$ there exists $\delta>0$ such that the following holds. Let $G$ be a graph with $n$ vertices and at most $(\frac{1}{4}-\epsilon)\frac{n^{2}}{2}$ edges such that $G$ is $(4\delta,C)$-dense and $\delta$-full. Then $G$ contains a weak-subdivision of $K_{5}$.	
\end{theorem}

In the rest of this section, we prove this theorem. Let us briefly outline the proof. With the help of the Regularity lemma, we partition our graph $G$ into a constant number of sets such that the bipartite graph induced by most pairs of these sets is random like. The notion of regularity and the Regularity lemma is described in Section \ref{sect:regularity}. Given a partial subdivision $H$ of $K_{5}$, we show that an induced weak-subdivision of $H$ can be found in $G$ if we can find $|V(H)|$ parts in the partition of $G$ satisfying certain properties. This argument is presented in Section \ref{sect:embedding}. Then, we finish our proof in Section \ref{sect:admissible} by showing that a regular partition of $G$ must contain $|V(H)|$ parts with the desired properties.

\subsection{Regularity Lemma}\label{sect:regularity}

In this section, we define the notion of regularity and state the Regularity lemma.  

If $G$ is a graph and $A,B\subset V(G)$, let $$d(A,B)=\frac{|E(A,B)|}{|A||B|}.$$
 For $\lambda>0$, a pair of subsets $(A,B)$ of $V(G)$ is \emph{$\lambda$-regular}, if for every $A'\subset A$ and $B'\subset B$ satisfying $|A'|\geq \lambda |A|$ and $|B'|\geq \lambda |B|$, we have $$|d(A,B)-d(A',B')|\leq \lambda.$$ 
 A $\lambda$-regular partition of the graph $G$ is a partition $V(G)= V_{1}\cup\dots\cup V_{k}$ such that 
 \begin{itemize}
 	\item $|V_{1}|,\dots,|V_{k}|\in \{\lfloor \frac{n}{k}\rfloor,\lceil \frac{n}{k}\rceil\}$,
 	\item all but at most $\lambda k^{2}$ of the pairs $(V_{i},V_{j})$ is $\lambda$-regular.
 \end{itemize} 

\begin{reglemma}
	For every $\lambda>0$ and positive integer $m$ there exists a positive integer $M$ such that the following holds. Let $G$ be a graph, then $G$ has a $\lambda$-regular partition into $k$ parts, where $m\leq k \leq M$.
\end{reglemma}

Given a graph $G$ with $\lambda$-regular partition $(V_{1},\dots,V_{k})$, the \emph{reduced graph} of this partition is the edge-weighted graph $(R,w)$, where $w:E(R)\rightarrow [0,1]$, is defined as follows.
\begin{itemize}
	\item The vertex set of $R$ is $[k]$,
	\item $i$ and $j$ are joined by an edge if $(V_{i},V_{j})$ is $\lambda$-regular,
	\item if $ij\in E(R)$, then $w(ij)=d(V_{i},V_{j})$.
\end{itemize}

%\begin{prop}
%	Let $G$ be a graph and let $(A,B)$ be a $\lambda$-regular pair. If $A'\subset A$ and $B'\subset B$ such that $|A'|\geq \epsilon |A|$ and $|B'|\geq \epsilon |B|$, then $(A',B')$ is $\lambda$
%\end{prop}

\subsection{Embedding}\label{sect:embedding}

Let $(R,w)$ be a complete edge-weighted graph, and let $\epsilon_{1}>0$. An edge $xy\in E(R)$ is \emph{$\epsilon_{1}$-thin} if $w(xy)\leq \lambda$, and $xy$ is \emph{$\epsilon_{1}$-fat} if $w(xy)\geq 1-\epsilon_{1}$. Let $H$ be a graph with $|V(R)|$ vertices. Say that $(R,w)$ is \emph{$(H,\epsilon_{1})$-admissible} if there exists a bijection $b:V(H)\rightarrow V(R)$ and a total ordering $\prec$ of the vertices of $H$ such that
\begin{itemize}
	\item no edge of $R$ is $\epsilon_{1}$-fat,
	\item if $xy\in E(H)$ such that $x\prec y$ and $b(x)b(y)$ is $\epsilon_{1}$-thin, then $w(b(x)b(z))+w(b(y)b(z))<1-\epsilon_{1}$ holds for every $z\in V(H)\setminus\{x,y\}$ satisfying $x\prec z$.
\end{itemize}  
Say that $b$ and $\prec$ \emph{witness} that $(R,w)$ is $(H,\epsilon_{1})$-admissible if they satisfy the properties above.

This section is devoted to the proof of the following embedding lemma.

\begin{lemma}\label{lemma:embedding}
	Let $H$ be a graph with $h$ vertices, let $k$ be a positive integer and let $\alpha,\beta,\delta,\lambda,\epsilon_{1}>0$ be real numbers satisfying $\epsilon_{1}<\beta$, $\lambda<\frac{\beta}{4h}(\frac{\epsilon_{1}}{2})^{2h^{2}}$, $\delta<\frac{1}{k}(\frac{\epsilon_{1}}{2})^{2h^{2}}$, and $\alpha<\frac{1}{k}(\frac{\epsilon_{1}}{2})^{2h}$. Let $G$ be an $(\alpha,\beta)$-dense, $\delta$-full graph with a $\lambda$-regular partition $(V_{1},\dots,V_{k})$ and corresponding reduced graph $(R,w)$. If $(R,w)$ contains an $(H,\epsilon_{1})$-admissible subgraph, then $G$ contains a weak-2-subdivision of $H$ as an induced subgraph. 
\end{lemma}

\begin{proof}
	Let $N=\lfloor \frac{n}{k}\rfloor$ and for simplicity, assume that $|V_{1}|=\dots=|V_{k}|=N$. Let the vertex set of $H$ be $\{v_{1},\dots,v_{h}\}$, and let $R'$ be $(H,\epsilon_{1})$-admissible subgraph of $R$. Without loss of generality, suppose that  $\{1,\dots,h\}$ is the vertex set of $R'$. Also, the bijection $b$ and ordering $\prec$ witnessing that $R'$ is $(H,\epsilon_{1})$-admissible are defined as follows:  $b(v_{i})=i$ for $i=1,\dots,h$, and $v_{1}\prec\dots\prec v_{h}$.
	
	 Let $H'$ be the 2-subdivision of $H$, and for every edge $v_{i}v_{j}\in E(H)$, let $s_{i,j}$ and $s_{j,i}$ be the two side-vertices in $H'$ 2-subdividing $v_{i}v_{j}$, where $s_{i,j}$ is joined to $v_{i}$, and $s_{j,i}$ is joined to $v_{j,i}$. 
	
	We embed the vertices of $H'$ in $G$ such that 
	\begin{itemize}
		\item for $i=1,\dots,h$, the image of $v_{i}$ is some vertex $v_{i}'\in V_{i}$, and for $v_{i}v_{j}\in E(H)$, the image of $s_{i,j}$ is some vertex $s_{i,j}'\in V_{i}$,
		\item the image of each edge of $H'$ is an edge,
		\item the image of each non-edge of $H'$ is a non-edge, with the exception that $s'_{i,j}$ and $s'_{i,j'}$ might be joined by an edge if $v_{i}v_{j},v_{i}v_{j'}\in E(H)$.
	\end{itemize}
	Clearly, such an embedding induces a weak-2-subdivision of $H$ in $G$, so our task is reduced to showing that such an embedding exists. See Figure \ref{figure2} for an illustration.

	\begin{figure}[t]
		\begin{center}
			\begin{tikzpicture}[scale=1]
			
			\draw (2,0) circle (1.2) ;
			\draw (0,2) circle (1.2) ;
			\draw (-2,0) circle (1.2) ;
			\draw (0,-2) circle (1.2) ;
			\node[] at (3.5,0) {$V_{1}$} ;
			\node[] at (0,3.5) {$V_{2}$} ;
			\node[] at (-3.5,0) {$V_{3}$} ;
			\node[] at (0,-3.5) {$V_{4}$} ;

			\node[vertex,minimum size=5pt] (A) at (2.3,0) {} ;
			\node[vertex,minimum size=5pt] (B) at (0,2.3) {} ;
			\node[vertex,minimum size=5pt] (C) at (-2.3,0) {} ;
			\node[vertex,minimum size=5pt] (D) at (0,-2.3) {} ;
			\node[] at (2.6,0) { $v_{1}'$} ;
			\node[] at (0,2.6) { $v_{2}'$} ;
			\node[] at (-2.6,0) { $v_{3}'$} ;
			\node[] at (0,-2.6) {$v_{4}'$} ;
			
			\node[vertex,minimum size=3pt] (AB) at (2,0.5) {} ; \node[] at (2.3,0.7) {$s_{1,2}'$} ;
			\node[vertex,minimum size=3pt] (AC) at (1.6,0.2) {} ; \node[] at (1.3,0.5) {$s_{1,3}'$} ;
			\node[vertex,minimum size=3pt] (AD) at (1.7,-0.5) {} ; \node[] at (2,-0.8) {$s_{1,4}'$} ;
			
			\node[vertex,minimum size=3pt] (BA) at (0.5,2) {} ; \node[] at (0.8,2.2) {$s_{2,1}'$} ;
			\node[vertex,minimum size=3pt] (BD) at (0.2,1.6) {} ; \node[] at (0.5,1.4) {$s_{2,4}'$} ;
			\node[vertex,minimum size=3pt] (BC) at (-0.5,1.7) {} ; \node[] at (-0.8,2) {$s_{2,3}'$} ;
			
			\node[vertex,minimum size=3pt] (CB) at (-2,0.5) {} ; \node[] at (-2.3,0.8) {$s_{3,2}'$} ;
			\node[vertex,minimum size=3pt] (CA) at (-1.6,0.2) {} ; \node[] at (-1.5,0.5) {$s_{3,1}'$} ;
			\node[vertex,minimum size=3pt] (CD) at (-1.7,-0.5) {} ; \node[] at (-1.9,-0.7) {$s_{3,4}'$} ;
			
			\node[vertex,minimum size=3pt] (DA) at (0.5,-2) {} ; \node[] at (0.8,-2.3) {$s_{4,1}'$} ;
			\node[vertex,minimum size=3pt] (DB) at (0.2,-1.6) {} ; \node[] at (-0.1,-1.3) {$s_{4,2}'$} ;
			\node[vertex,minimum size=3pt] (DC) at (-0.5,-1.7) {} ; \node[] at (-0.7,-2) {$s_{4,3}'$} ;
			
			\draw (A) -- (AB) -- (BA) -- (B) ;
			\draw (B) -- (BC) -- (CB) -- (C) ;
			\draw (C) -- (CD) -- (DC) -- (D) ;
			\draw (D) -- (DA) -- (AD) -- (A) ;
			\draw (A) -- (AC) -- (CA) -- (C) ;
			\draw (B) -- (BD) -- (DB) -- (D) ;
			
			\draw[red] (AB) -- (AC) -- (AD) ;
			\draw[red] (BC) -- (BA) -- (BD) -- (BC) ;
			\draw[red] (CA) -- (CD)  ;
			\draw[red] (DA) -- (DC) -- (DB) ;
			
			\end{tikzpicture}
			\caption{An embedding of a weak-2-subdivision of $K_{4}$. The edges that are part of the weak-subdivision but not the subdivision are red.}
			\label{figure2}
		\end{center}
	\end{figure}
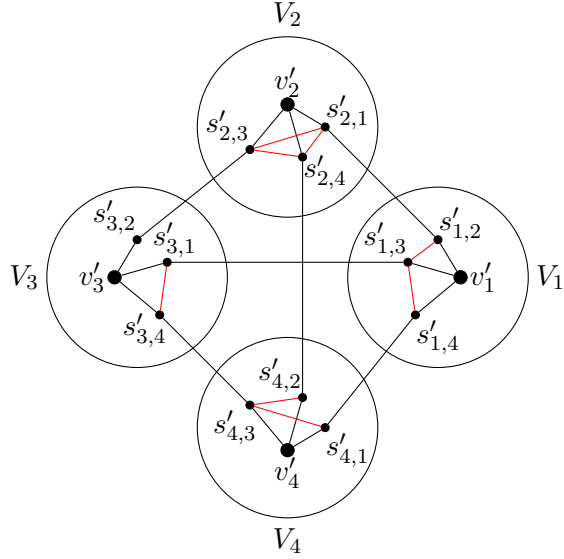

	Given a vertex $v\in V_{i}$ for some $i\in [h]$ and $h-1$ sets $U_{j}\subset V_{j}$ for $j\in [h]\setminus\{i\}$, say that \emph{$v$ is average with respect to $(U_{j})_{j\in [h]\setminus\{i\}}$} if 
	$$\left|\frac{|N(v)\cap U_{j}|}{|U_{j}|}-w(ij)\right|<\lambda$$
	for $j\in [h]\setminus\{i\}$. The following useful claim is an easy consequence of the regularity condition. 
	
	\begin{claim}\label{claim:degree}
		Let $i\in [h]$, and for $j\in [h]\setminus \{i\}$, let $U_{j}\subset V_{j}$ such that $|U_{j}|\geq \lambda N$. Then the number of vertices in $V_{i}$ which are not average with respect to $(U_{j})_{j\in[h]\setminus \{i\}}$ is at most $2h\lambda N$. 
	\end{claim}

    \begin{proof}
    	Let $W\subset V_{i}$ be the set of vertices  which are not average with respect to $(U_{j})_{j\in[h]\setminus \{i\}}$. If $|W|\geq 2h\lambda N$, then there exists $j\in [h]\setminus \{i\}$ and $W'\subset W$ such that $|W'|\geq \lambda N$, and either $\frac{|N(v)\cap U_{j}|}{|U_{j}|}>w(ij)+\lambda$ for every $v\in W'$ or $\frac{|N(v)\cap U_{j}|}{|U_{j}|}<w(ij)-\lambda$ for every $v\in W'$. In the first case $d(W',U_{j})>d(V_{i},V_{j})+\lambda$, in the second case $d(W',U_{j})<d(V_{i},V_{j})-\lambda$, both contradicting that $(V_{i},V_{j})$ is a $\lambda$-regular pair.
    \end{proof}
    
    For $i=1,\dots,h$, set $U_{i}:=V_{i}$. We embed the vertices of $H'$ in $G$ step-by-step while updating the sets $U_{1},\dots,U_{h}$ at the end of each step. During this procedure, say that an index $i\in [h]$ is \emph{active}, if either $v_{i}'$ or $s_{i,j}'$ is not defined yet for some $v_{i}v_{j}\in E(H)$. During our procedure, the following properties are satisfied:
    \begin{itemize}
    	\item at each step, we embed either one branch-vertex or two side-vertices connected by an edge,
    	\item when a set is updated, it is replaced with one of its subsets,
    	\item at the end of step $s$, $|U_{i}|\geq (\frac{\epsilon_{1}}{2})^{2s}N$ for $i\in [h]$,
    	\item at the end of each step, if $v\in V(H')$ is already embedded in $G$ and its image is $v'\in V_{i}$, then for every active index $j\in [h]\setminus i$ we have  $N(v')\cap U_{j}=\emptyset$.
    \end{itemize}
     In the first $h$ steps, we embed the vertices $v_{1},\dots,v_{h}$ one-by-one. More precisely, for $s=1,\dots,h$, at step $s$ we define the image $v'_{s}\in V_{s}$ of $v_{s}$ and update the sets $U_{1},\dots,U_{h}$ such that in addition to the properties above, the following condition is also satisfied:
     \begin{itemize}
     	\item at the end of step $s$, we have $U_{s}\subset N(v'_{s})$.
     \end{itemize}
     
    We proceed at step $s$ as follows. In the beginning of step $s$, we have $|U_{i}|\geq (\frac{\epsilon_{1}}{2})^{2s-2}N$ for $i\in [h]$. As $G$ is $(\alpha,\beta)$-dense and $|U_{s}|\geq (\frac{\epsilon_{1}}{2})^{2h}N\geq \alpha n$, we have $d(G[U_{s}])\geq \beta$. Let $W$ be the set of vertices in $U_{s}$ whose degree in $G[U_{s}]$ is at least $\beta|U_{s}|$. Then 
    $$\beta |U_{s}|^{2}\leq E(G[U_{s}])\leq  \frac{1}{2}(|W||U_{s}|+\beta |U_{s}|^{2}),$$
    so $|W|\geq \frac{\beta}{2}|U_{s}|\geq \frac{\beta}{2}(\frac{\epsilon_{1}}{2})^{2h}N$. As $|W|> 2h\lambda N$, there exists $w\in W$ such that $w$ is average with respect to $(U_{i})_{i\in [h]\setminus \{s\}}$. Note that for $i\in [h]\setminus \{s\}$, we have
    $$|U_{i}\setminus N(w)|=|U_{i}|-|U_{i}\cap N(w)|\geq |U_{i}|-|U_{i}|(w(ij)+\lambda)\geq |U_{i}|(\epsilon_{1}-\lambda)>\frac{\epsilon_{1}}{2}|U_{i}|> \left(\frac{\epsilon_{1}}{2}\right)^{2s}N,$$
    where the second inequality holds noting that $ij$ is not $\epsilon_{1}$-fat. Also,
    $$|U_{s}\cap N(w)|\geq \beta|U_{s}|\geq \epsilon_{1}|U_{s}|>\left(\frac{\epsilon_{1}}{2}\right)^{2s}N.$$
   
    Set $v'_{s}=w$, and make the following updates: $U_{i}:=U_{i}\setminus N(w)$ for $i\in [h]\setminus \{s\}$ and $U_{s}:=N(w)\cap U_{s}$. This finishes step $s$.
    
    Let $f_{1},\dots,f_{l}$ be an enumeration of the edges of $H$ such that if $1\leq i<j\leq l$, and $f_{i}=v_{a}v_{b}$, $f_{j}=v_{a'}v_{b'}$, where $a<b$, $a'<b'$, then $a\leq a'$. Now for $s=1,\dots,l$, we embed the two side-vertices on the $2$-subdivision of the edge $f_{s}$ in step $h+s$. 
    
    We proceed at step $h+s$ as follows. In the beginning of the step, we have $|U_{i}|\geq \left(\frac{\epsilon_{1}}{2}\right)^{2(h+s)-2}N$. Let $f_{s}=v_{a}v_{b}$, where $a<b$. Note that by the ordering of the edges, the active indices at this step are $i\in \{a,a+1,\dots,h\}$. Consider two cases.
    \begin{description}
    	\item[Case 1.] $w(ab)\geq \epsilon_{1}$. 
    	
    	As $|U_{a}|>(\frac{\epsilon_{1}}{2})^{2h^{2}}N>2h\lambda N$, there exists a vertex $u\in U_{a}$ such that $u$ is average with respect to $(U_{i})_{i\in [h]\setminus \{a\}}$. For $i\in [h]\setminus \{a\}$, let $U_{i}'=U_{i}\setminus N(u)$, and let $U_{a}'=U_{a}$. Then, by similar calculations as before, we have $|U_{i}'|\geq \frac{\epsilon_{1}}{2}|U_{i}|$. Also, let $W=U_{b}\cap N(u)$, then
    	$$|W|=|U_{b}\cap N(u)|\geq |U_{b}|(w(ab)-\lambda)\geq  |U_{b}|(\epsilon_{1}-\lambda)\geq \frac{\epsilon_{1}}{2}|U_{b}|.$$
    	But then $|W|> 2h\lambda N$, so there exists $w\in W$ such that $w$ is average with respect to $(U'_{i})_{i\in [h]\setminus \{b\}}$. But then for $i\in [h]\setminus \{b\}$, by the same calculations as before, we have 
    	$$|U_{i}'\setminus N(w)|\geq \frac{\epsilon_{1}}{2}|U_{i}'|\geq \left(\frac{\epsilon_{1}}{2}\right)^{2}|U_{i}|\geq \left(\frac{\epsilon_{1}}{2}\right)^{2(h+s)}N.$$
    	Set $s_{a,b}'=u$, $s_{b,a}'=w$, and make the following updates: $U_{i}:=U_{i}'\setminus N(w)$ for $i\in [h]\setminus \{b\}$ and $U_{b}:=U_{b}'$. This finishes step $h+s$.
    	
    	\item[Case 2.]$w(ab)<\epsilon_{1}$.
    	
    	Let $W_{a}$ be the set of elements in $U_{a}$ that are average with respect to $(U_{i})_{i\in [h]\setminus \{a\}}$, and let $W_{b}$ be the set of elements in $U_{b}$ that are average with respect to $(U_{i})_{i\in [h]\setminus \{b\}}$. Then $|W_{a}|\geq |U_{a}|-2h\lambda N\geq \frac{1}{2}|U_{a}|$, and similarly $|W_{b}|\geq \frac{1}{2}|U_{b}|$. But then $|W_{a}|,|W_{b}|>(\frac{\epsilon_{1}}{2})^{2h^{2}}N>\delta n$ and $G$ is $\delta$-full, so there exists an edge between $W_{a}$ and $W_{b}$. Let $w_{a}\in W_{a}$ and $w_{b}\in W_{b}$ be the endpoints of such an edge. For every $i\in \{a+1,\dots,h\}\setminus \{b\}$, let $U_{i}'=U_{i}\setminus (N(w_{a})\cup N(w_{b}))$. Then
    	$$|U_{i}'|=|U_{i}|-|N(w_{a})\cup N(w_{b})|\geq |U_{i}|-|N(w_{a})|-|N(w_{b})|\geq |U_{i}|-(w(ai)+\lambda)|U_{i}|-(w(bi)+\lambda)|U_{i}|.$$
    	Note that we have $w(ai)+w(bi)\leq 1-\epsilon_{1}$ as $R'$ is $(H,\epsilon_{1})$-admissible and $a<i$. Hence,
    	$$|U_{i}'|\geq |U_{i}|(\epsilon_{1}-2\lambda)\geq \frac{\epsilon_{1}}{2}|U_{i}|\geq \left(\frac{\epsilon_{1}}{2}\right)^{2(h+s)}N.$$
    	Set $s_{a,b}'=w_{a}$, $s_{b,a}'=w_{b}$, and make the following updates: $U_{i}:=U_{i}'$ if $i\in \{a+1,\dots,h\}\setminus \{b\}$, $U_{a}:=U_{a}\setminus N(w_{b})$, $U_{b}:=U_{b}\setminus N(w_{a})$, and do not update the sets $U_{1},\dots,U_{a-1}$. This finishes step $h+s$.
    \end{description}
 
    It is easy to check that the end of the procedure we get the desired embedding of $H'$ in $G$.
\end{proof}

\subsection{Finding an admissible subgraph}\label{sect:admissible}

Given an edge weighted graph $(R,w)$, the \emph{weight} of $R$ is defined as $w(R)=\sum_{f\in E(R)}w(f)$. The aim of this section is to prove the following lemma.

\begin{lemma}\label{lemma:admissible1}
	Let $\mathcal{H}$ be the family of partial subdivisions of $K_{5}$ with at most $8$ vertices. For every $\epsilon_{2}>0$ there exist $k_{0}=k_{0}(\epsilon_{2})\in \mathbb{N}^{+}$ and $\epsilon_{1}=\epsilon_{1}(\epsilon_{2})>0$ such that the following holds for every $k\geq k_{0}$. Let $(R,w)$ be an edge weighted graph with $k$ vertices, at least $(1-\epsilon_{1})\frac{k^{2}}{2}$ edges such that $w(R)\leq (\frac{1}{4}-\epsilon_{2})\frac{k^{2}}{2}$. Then there exists $H\in \mathcal{H}$ such that $R$ contains an $(H,\epsilon_{1})$-admissible subgraph. 
\end{lemma}

Given the edge weighted graph $(R,w)$, let $R_{0}$ be the \emph{complete} graph on $V(R)$, and define the weight function $w_{0}:E(R_{0})\rightarrow [0,1]$ as follows. For every $f\in E(R_{0})$, let 
$$w_{0}(f)=\begin{cases}
\max(w(f)+\epsilon_{1},1) &\mbox{ if } f\in E(R)\mbox{ and }w(f)> \epsilon_{1},\\
0 &\mbox{ if }f\in E(R)\mbox{ and }w(f)\leq \epsilon_{1},\\
1 &\mbox{ if }f\not\in E(R).
\end{cases}$$
Then it is easy to see that for every graph $H$, an $(H,0)$-admissible subgraph of $(R_{0},w_{0})$ is $(H,\epsilon_{1})$-admissible in $(R,w)$. Also, if $|E(R)|\geq (1-\epsilon_{1})\frac{|V(R)|^{2}}{2}$, then $|w(R)-w_{0}(R_{0})|<2\epsilon_{1}|V(R)|^{2}$. Therefore, in order to prove Lemma \ref{lemma:admissible1}, it is enough to prove the following lemma. 

\begin{lemma}\label{lemma:admissible2}
		Let $\mathcal{H}$ be the family of partial subdivisions of $K_{5}$ with at most $8$ vertices. For every $\epsilon_{3}>0$ there exist $k_{1}=k_{1}(\epsilon_{3})\in \mathbb{N}^{+}$ such that the following holds for every $k\geq k_{1}$. Let $(R,w)$ be a complete edge weighted graph with $k$ vertices such that $w(R)\leq (\frac{1}{4}-\epsilon_{3})\frac{k^{2}}{2}$. Then there exists $H\in \mathcal{H}$ such that $R$ contains an $(H,0)$-admissible subgraph.
\end{lemma}

\begin{proof}[Proof of Lemma \ref{lemma:admissible1} assuming Lemma \ref{lemma:admissible2}]
 Let $\epsilon_{3}=\frac{\epsilon_{2}}{2}$ and let $k_{1}=k_{1}(\epsilon_{3})$ be the constant given by Lemma \ref{lemma:admissible2}. We show that the choice $k_{0}=k_{1}$ and $\epsilon_{1}=\frac{\epsilon_{2}}{8}$ suffices. Let $(R,w)$ be an  edge weighted graph with $k\geq k_{0}$ vertices, at least $(1-\epsilon_{1})\frac{k^{2}}{2}$ edges such that $w(R)\leq (\frac{1}{4}-\epsilon_{2})\frac{k^{2}}{2}$, and define the complete edge weighted graph $(R_{0},w_{0})$ as above. Then $w_{0}(R_{0})\leq w(R)+2\epsilon_{1}k^{2}\leq (\frac{1}{4}-\epsilon_{3})\frac{k^{2}}{2}.$ Hence, $(R_{0},w_{0})$ contains an $(H,0)$-admissible subgraph for some $H\in \mathcal{H}$. But then this subgraph is also $(H,\epsilon_{1})$-admissible in $(R,w)$.
\end{proof}

As a reminder, a complete subgraph $R'$ of the complete edge-weighted graph $(R,w)$ is $(H,0)$-admissible for some graph $H$ if $|V(R')|=|V(H)|$, and there exist a bijection $b:V(H)\rightarrow V(R')$ and a total ordering $\prec$ of the vertices of $H$ such that 
\begin{itemize}
	\item $R'$ contains no edge of weight $1$,
	\item if $xy\in E(H)$ such that $x\prec y$ and $w(b(x)b(y))=0$, then $w(b(x)b(y))+w(b(x)b(z))<1$ holds for every $z\in V(H)\setminus \{x,y\}$ satisfying $x\prec y$.
\end{itemize}

In what comes, fix a positive integer $k$ and a finite family of graphs $\mathcal{H}$. Let $(R,w)$ be a \emph{complete} edge-weighted graph on $k$ vertices that does not contain an $(H,0)$-admissible subgraph for any $H\in \mathcal{H}$. Also, among every such weighted graph, choose the one that minimizes the weight $w(R)$, and among the ones that minimize the weight, it minimizes the size of the set $$F_{w}=\{w(f):f\in E(R)\}\setminus\left\{0,\frac{1}{2},1\right\}.$$
Note that the set of weight functions $w:E(R)\rightarrow [0,1]$ for which $(R,w)$ does not contain an $(H,0)$-admissible subgraph for any $H\in \mathcal{H}$ is a compact subset of $\mathbb{R}^{E(R)}$, so $w(R)$ attains its minimum on this set.  We show that $(R,w)$ must have a simple structure, reminiscent of the extremal graphs given by the classical Tur\'an's theorem \cite{T41} not containing a clique of given size. To show this, we proceed by following a similar train of thoughts as one of the traditional proofs of Tur\'an's theorem.

\begin{prop}\label{prop:diffweights}
  $F_{w}=\emptyset.$
\end{prop}

\begin{proof}
	 Say that a triple of vertices $(x,y,z)$ in $(R,w)$ is \emph{dangerous} if $w(xy)=0$ and $w(xz)+w(yz)\geq 1$, and let $D_{w}$ be the set of dangerous triples. We show that if $F_{w}$ is nonempty, then we can replace the weight function $w$ with $w'$ such that $w'(R)\leq w(R)$, $|F_{w'}|<|F_{w}|$ and $D_{w}\subset D_{w'}$. But whether a subgraph of $(R,w)$ is $(H,0)$-admissible depends only on the set of dangerous triples, so if $(R,w)$ is not $(H,0)$-admissible for any $H\in\mathcal{H}$, then neither is $(R,w')$.
	 
	For $r\in [0,1]$, let $F(r)=\{f\in E(R):w(f)=r\}$. Suppose that $F_{w}$ is nonempty and let $r\in F_{w}$. Consider two cases.
    \begin{description}
    	\item[Case 1.]$1-r\not\in F_{w}$. If $r<\frac{1}{2}$, let $q$ be the largest real number smaller than $r$ such that  $q\in F_{w}$ or $1-q\in F_{w}$; if there exists no such $q$, let $q=0$. If $r>\frac{1}{2}$, let $q$ be the largest real number smaller than $r$ but not less than $\frac{1}{2}$ such that  $q\in F_{w}$ or $1-q\in F_{w}$; if there exists no such $q$, let $q=\frac{1}{2}$. Define the new weight function $w'$ as 
    	$$w'(f)=\begin{cases}q &\mbox{ if }f\in F(r),\\ 
    	w(f)&\mbox{ if }f\in E(R)\setminus F(r).\end{cases}$$
    	Then $w'(R)=w(R)-(r-q)|F(r)|<w(R)$, $F_{w'}=F_{w}\setminus\{r\}$ and $D_{w}\subset D_{w'}$.
    	
    	\item[Case 2.] $1-r\in F_{w}$. Without loss of generality, we can suppose that $|F(r)|\geq |F(1-r)|$. Again, if $r<\frac{1}{2}$, let $q$ be the largest real number smaller than $r$ such that  $q\in F_{w}$ or $1-q\in F_{w}$; if there exists no such $q$, let $q=0$. If $r>\frac{1}{2}$, let $q$ be the largest real number smaller than $r$ but not less than $\frac{1}{2}$ such that  $q\in F_{w}$ or $1-q\in F_{w}$; if there exists no such $q$, let $q=\frac{1}{2}$. Define the new weight function $w'$ as 
    	$$w'(f)=\begin{cases}q &\mbox{ if }f\in F(r),\\ 
    	1-q &\mbox{ if }f\in F(1-r),\\
    	w(f)&\mbox{ if }f\in E(R)\setminus (F(r)\cup F(1-r)).\end{cases}$$ 
    	Then $w'(R)=w(R)+(q-r)(|F(r)|-|F(1-r)|)\leq w(R)$, $F_{w'}=F_{w}\setminus\{r,1-r\}$ and $D_{w}\subset D_{w'}$.
    \end{description}

\end{proof}

\begin{prop}\label{prop:partition}
	Among the complete edge weighted graphs $(R,w)$ on $k$ vertices, where $w:E(R)\rightarrow \{0,\frac{1}{2},1\}$ and $(R,w)$ has no $(H,0)$-admissible subgraph for any $H\in \mathcal{H}$, there is one with minimum weight $w(R)$ satisfying the following properties. The vertex set of $R$ can be partitioned into some parts $X_{1},\dots,X_{s}$ such that for $i\in [s]$, every edge in $X_{i}$ has weight $1$, and for every $1\leq i<j\leq s$, either every edge between $X_{i}$ and $X_{j}$ has weight $\frac{1}{2}$, or every edge between $X_{i}$ and $X_{j}$ has weight $0$.
\end{prop}

\begin{proof}
	Let $(R,w)$ be a complete edge weighted graph satisfying the desired conditions of the proposition. For $x\in V(R)$, let $d_{w}(x)=\sum_{y\in V(R)\setminus \{x\}} w(xy)$. For $x,y\in V(R)$, let $x\sim y$ if $w(xy)=1$. We show that $\sim$ is an equivalence relation. Suppose not, then there exists $x,y,z\in V(R)$ such that $x\sim y$, $y\sim z$ but $x\not\sim z$. Consider two cases.
	
	\begin{description}
		\item[Case 1.] Either $d_{w}(y)>d_{w}(x)$ or $d_{w}(y)>d_{w}(z)$.
		
		Suppose that $d_{w}(y)>d_{w}(x)$, the other case being similar. Define the new weight function
		$$w'(f)=\begin{cases}w(xu) &\mbox{ if }f=yu\mbox{ for some }u\in V(R),\\ 
		w(f) &\mbox{ otherwise.}\end{cases}$$ 
		Then $w'(R)=w(R)+d_{w}(x)-d_{w}(y)<w(R)$ and $(R,w')$ does not contain an $(H,0)$-admissible subgraph for any $H\in\mathcal{H}$. Indeed, $(R,w)$ and $(R',w)$ differ only in the edges containing $y$. Hence, if $(R,w')$ contains an $(H,0)$-admissible subgraph $R'$, then $y\in V(R')$. But as $w(xy)=1$, we must have $x\not\in V(R')$. The neighborhoods of $x$ and $y$ are identical in $(R,w')$, so then $R'\setminus\{y\}\cup\{x\}$ is also $(H,0)$-admissible in both $(R,w)$ and $(R,w')$, contradiction.
		
		\item[Case 2.]  $d_{w}(y)\leq d_{w}(x)$ and $d_{w}(y)\leq d_{w}(z)$.
		
		Define the new weight function
			$$w'(f)=\begin{cases}w(yu) &\mbox{ if }f=xu\mbox{ or }f=zu\mbox{ for some }u\in V(R),\\ 
		w(f) &\mbox{ otherwise.}\end{cases}$$ 
		
			Then $w'(R)=w(R)+(2d_{w}(y)-1)-(d_{w}(x)+d_{w}(z)-w(xz))<w(R)$. Also, by a similar argument as in the previous case,  $(R,w')$ does not contain an $(H,0)$-admissible subgraph for any $H\in\mathcal{H}$
	\end{description}
	
	Both cases contradict the minimality of $w(R)$, so $\sim$ is an equivalence relation. Let $X_{1},\dots,X_{s}$ be the equivalence classes of $\sim$. 
	
	Now consider the following operation on weight functions. Let $w_{0}:E(R)\rightarrow \mathbb{R}^{+}$ be a weight function such that $w_{0}(xy)=1$ for every $x,y\in X_{i}$, $i=1,\dots,s$. Define the weight function $C_{i}w_{0}$ on $E(R)$ as follows. Let $x\in X_{i}$ be a vertex such that $d_{w_{0}}(x)$ is minimal among the vertices in $X_{i}$. Then for $f\in E(R)$, let 
	$$C_{i}w_{0}(f)=\begin{cases}w_{0}(xz) &\mbox{ if }f=yz\mbox{ for some }y\in X_{i},u\in V(R),\\ 
	w_{0}(f) &\mbox{ otherwise.}\end{cases}$$ 
	
	Then $C_{i}w_{0}(R)=w_{0}(R)+|X_{i}|d_{w_{0}}(x)-\sum_{y\in X_{i}}d_{w_{0}}(y)\leq w_{0}(R)$. Also, for every graph $H$, if $(R,w_{0})$ does not have an $(H,0)$-admissible subgraph, then $(R,C_{i}w_{0})$ also does not contain an $(H,0)$-admissible subgraph. This holds by the same argument as in Case 1. above.
	
	Let $w'=C_{s}C_{s-1}\dots C_{1}w$. Then $w'(R)\leq w(R)$ and $(R,w')$ does not contain an $(H,0)$-admissible subgraph for any $H\in\mathcal{H}$. Moreover, for any $1\leq i<j\leq s$, either every edge between $X_{i}$ and $X_{j}$ has weight $0$, or every edge between $X_{i}$ and $X_{j}$ has weight $\frac{1}{2}$. Indeed, it is easy to prove by induction on $l$ that if $w_{l}=C_{l}C_{l-1}\dots C_{1}w$, then for every pair $(i,u)$, where  $1\leq i\leq l$ and $u\in X_{j}$ for some $j\in [s]\setminus \{i\}$, either for every $v\in X_{i}$ we have $w_{l}(uv)=0$, or for every $v\in X_{i}$ we have $w_{l}(uv)=\frac{1}{2}$. But then $w'=w_{s}$ has the desired properties.
\end{proof}

In what comes, suppose that $(R,w)$ has the form described in Proposition \ref{prop:partition}. Define the \emph{vertex}-weighted graph $(Q,\phi)$, where $\phi:V(Q)\rightarrow [0,1]$, as follows. Let $X_{1},\dots,X_{s}$ be the partition of $V(R)$ given by Proposition \ref{prop:partition}. Then the vertex set of $Q$ is $[s]$, the weight of $a\in [s]$ is $\phi(a)=\frac{|X_{a}|}{k}$, and $ab$ is an edge of $Q$ if every edge of $R$ between $X_{a}$ and $X_{b}$ has weight $\frac{1}{2}$. Note that 
$$\sum_{a\in [s]}\phi(a)=1.$$
 We define the \emph{weight} of the graph $Q$ in the following unconventional way: 
$$\phi(Q)=\sum_{a\in [s]} \phi(a)^{2}+\sum_{ab\in E(Q)}\phi(a)\phi(b).$$
Note that with this definition, we have
\begin{equation}\label{equ1}
w(R)=\sum_{xy\in E(R)}w(xy)=\sum_{i=1}^{s}\binom{|X_{i}|}{2}+\sum_{ab\in E(Q)}\frac{1}{2}|X_{a}||X_{b}|=\frac{k^{2}}{2}\left(\phi(Q)-\frac{1}{k}\right).
\end{equation}

Let $H$ be a graph and let $Q'$ be an induced subgraph of $Q$ on $|V(H)|$ vertices. Say that $Q'$ is \emph{$H$-admissible} if there exists a bijection $b:V(H)\rightarrow V(Q')$ and a total ordering $\prec$ of the vertices of $H$ such that
\begin{itemize}
\item if $xy\in E(H)$ such that $x\prec y$ and $b(x)b(y)\not\in E(Q')$, then for every $z\in V(H)\setminus\{x,y\}$ satisfying $x\prec z$, either $b(x)b(z)\not\in E(Q')$ or $b(y)b(z)\not\in E(Q')$   
\end{itemize}

Say that $b$ and $\prec$ \emph{witness} that $Q'$ is $H$-admissible if they satisfy the properties above. Note that $Q$ contains an $H$-admissible subgraph if and only if $(R,w)$ contains an $(H,0)$-admissible subgraph. Hence, $Q$ does not contain an $H$-admissible subgraph for any $H\in\mathcal{H}$. The following proposition is finishing touch in the proof of Lemma \ref{lemma:admissible2}.

\begin{prop}\label{prop:1/4}
	Let $\mathcal{H}$ be the family of partial subdivisions of $K_{5}$ with at most 8 vertices. If $(Q,\phi)$ has no $H$-admissible subgraph for any $H\in \mathcal{H}$, then $\phi(Q)\geq \frac{1}{4}.$
\end{prop}

\begin{proof} In order to make the proof of this proposition more transparent, we collect a few simple observations about $Q$.
	 
	 \begin{obs}\label{obs:basic}
	 	Any subset of 5 vertices of $Q$ contains two neighboring edges and two disjoint edges.
	 \end{obs}
 
     \begin{proof}
     	Let $A$ be a 5 element subset of $V(Q)$. If $Q[A]$ does not contain $2$ neighboring edges, then any ordering $\prec$ of the vertices of $K_{5}$ and any bijection $b:V(K_{5})\rightarrow A$ witness that $Q[A]$ is $K_{5}$-admissible.
     	
     	Now suppose that $A$ does not contain two disjoint edges. Then either $Q[A]$ is the union of a triangle and two isolated vertices, or $Q[A]$ is the union of a star and some isolated vertices. In the first case, any ordering $\prec$ of the vertices of $K_{5}$ and any bijection $b:V(K_{5})\rightarrow A$ witness that $Q[A]$ is $K_{5}$-admissible. In the second case, let $1\prec 2\prec 3\prec 4\prec 5$ be the vertices of $K_{5}$, and let $v\in A$ be the unique vertex with degree at least two. Then any bijection $b:V(K_{5})\rightarrow A$ satisfying $b(1)=v$ with the ordering $\prec$ witness that $Q[A]$ is $K_{5}$-admissible.
     \end{proof}
	
	\begin{obs}\label{obs:K23}
		$V(Q)$ does not contain two vertices $a,b$ and a set $A\subset V(Q)\setminus\{a,b\}$ such that $|A|=3$, and either
		
		\begin{enumerate}
			\item there are no edges between $\{a,b\}$ and $A$,
			\item there are no edges between $a$ and $A$, and $b$ is joined to every vertex of $A$,
			\item $ab$ is an edge, and $a$ and $b$ are joined to every element of $A$.
		\end{enumerate} 
	\end{obs}
    \begin{proof}
    	Let the vertices of $K_{5}$ be $1\prec 2\prec 3\prec 4\prec 5$. Suppose that $Q$ contains $a,b,A$ with one of the described properties, and let $A=\{x_{1},x_{2},x_{3}\}$. If $Q[A]$ contains exactly one non-edge, suppose that it is $x_{2}x_{3}$. Define the bijection $b:V(K_{5})\rightarrow A\cup\{a,b\}$ as follows. Let $b(1)=b,b(2)=a,b(3)=x_{1},b(4)=x_{2},b(5)=x_{3}$. Then $b$ and $\prec$ witness that $Q[A\cup\{a,b\}]$ is $K_{5}$-admissible.
    \end{proof}

    \begin{obs}\label{obs:K4}
    	Let $v\in V(Q)$ and $A\subset V(Q)\setminus \{v\}$ such that $|A|=4$ and either $v$ is joined to every element of $A$, or $v$ is joined to no element of $A$. Then $Q[A]$ is either a cycle of length 4 or a path of length 3.
    \end{obs}

    \begin{proof}
    	It is easy to check that $Q[A]$ is $K_{4}$-admissible, unless $Q[A]$ is a cycle of length 4 or a path of length 3. Let $1\prec 2\prec 3\prec 4$ be the vertices of $K_{4}$ and let $b:K_{4}\rightarrow A$ be a bijection such that $\prec$ and $b$ witness that $Q[A]$ is $K_{4}$-admissible. Then extending $\prec$ to $V(K_{5})$ as $5\prec 1\prec 2\prec 3\prec 4$ and the bijection $b$ as $b(5)=v$, $b$ and $\prec$ witness that $Q[A\cup \{v\}]$ is $K_{5}$-admissible. 
    \end{proof}

    \begin{obs}\label{obs:deg5}
    	If $s\geq 7$, the complement of $Q$ has maximum degree $4$.
    \end{obs}

    \begin{proof}
    	Assume that there exists $a\in V(Q)$ and $B\subset V(Q)\setminus\{v\}$ such that $|B|=5$ and there are no edges between $a$ and $B$. Let $b\in V(Q)\setminus (B\cup \{a\})$. Then there exists $A\subset B$ such that $|A|=3$ and either $b$ is joined to every element of $A$, or there are no edges between $b$ and $A$. This contradicts either 1. or 2. in Observation \ref{obs:K23}.
    \end{proof}

Now we show that $Q$ cannot have many vertices.

  \begin{claim}\label{claim:s<8}
  	$s<8$
  \end{claim} 

\begin{proof}
	Suppose that $s=8$. By Observation \ref{obs:deg5}, the minimum degree of $Q$ is at least $3$. First, suppose that there exists $v\in V(Q)$ of degree 3, and let $A$ be the set of 4 vertices not joined to $v$. Then by Observation \ref{obs:K4}, $Q[A]$ is either a path of length 3, or a cycle of length 4. In both cases, there are two non-edges $ab$ in $Q[A]$ for which there exists $c\in A\setminus \{a,b\}$ such that $ac,bc\in E(Q)$. Fix one of these non-edges $ab$, and let $B=(V(Q)\setminus (A\cup \{v\}))\cup \{a,b\}$. Then $|B|=5$, and as the minimum degree of $Q$ is 3, both $a$ and $b$ has degree at least $1$ in $Q[B]$. 
	
	We claim that there exists a path between $a$ and $b$ in $Q[B]$. Indeed, suppose not and let $C$ be the connected component of $Q[B]$ containing $a$. Then $|C|\geq 2$ and $|B\setminus C|\geq 2$ as both $a$ and $b$ has degree at least $1$. But then $\{|C|,|B\setminus C|\}=\{2,3\}$, and there are no edges between $B$ and $C$, which contradicts 1. in Observation \ref{obs:K23}. See Figure \ref{figure3}.
	
	Let $a=x_{0},x_{1},\dots,x_{p+1}=b$ be the vertices of a path connecting $a$ and $b$ in $B$, and set $D=A\cup \{v,x_{1},\dots,x_{p}\}$. Also, let $K_{5}^{(p)}$ be the partial subdivision of $K_{5}$ in which one edge is $p$-subdivided. We show that $D$ is $K_{5}^{(p)}$-admissible. Let $1,2,3,4,5$ be the branch-vertices of $K_{5}^{(p)}$, and let $r_{1},\dots,r_{p}$ be the vertices $p$-subdividing the edge $23$. Let $c,d$ be the two vertices in $A\setminus \{a,b\}$. Consider the vertex ordering on $K_{5}^{(p)}$ defined as $r_{1}\prec\dots\prec r_{p}\prec 1\prec 2\prec 3\prec 4\prec 5$, and define the bijection $b$ as $b(r_{i})=x_{i}$ for $i=1,\dots,p$, $b(1)=v, b(2)=a, b(3)=b, b(4)=c, b(5)=d$. Then $\prec$ and $b$ witness that $Q[D]$ is $K_{5}^{(p)}$-admissible.
	
	Now we can suppose that every vertex of $Q$ has degree at least $4$. Let $v\in V(Q)$ be an arbitrary vertex and let $A\subset N(v)$ such that $|A|=4$. Again, by Observation \ref{obs:K4}, we have that $Q[A]$ is either a cycle of length 4 or a path of length 3. Let $B=(V(Q)\setminus (A\cup \{v\}))\cup \{a,b\}$. Then $|B|=5$, and as the minimum degree of $Q$ is 4, both $a$ and $b$ have degree at least $1$ in $Q[B]$. But then we can repeat the exact same argument as before to find a $K_{5}^{(p)}$-admissible subgraph for some $p\leq 3$. This finishes the proof of the claim.
\end{proof}

	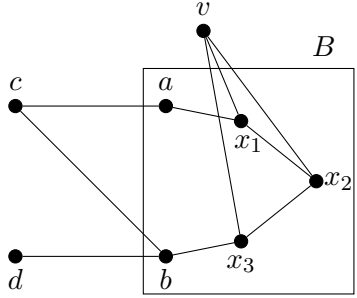
\begin{figure}[t]
	\begin{center}
		\begin{tikzpicture}[scale=1]

		\node[vertex,minimum size=5pt] (v) at (0,2) {} ; \node[] at (0,2.3) {$v$} ;
		\node[vertex,minimum size=5pt] (a) at (-0.5,1) {} ; \node[] at (-0.5,1.3) {$a$} ;
		\node[vertex,minimum size=5pt] (c) at (-2.5,1) {} ; \node[] at (-2.5,1.3) {$c$} ;
		\node[vertex,minimum size=5pt] (b) at (-0.5,-1) {} ; \node[] at (-0.5,-1.3) {$b$} ;
		\node[vertex,minimum size=5pt] (d) at (-2.5,-1) {} ; \node[] at (-2.5,-1.3) {$d$} ;
		
		\node[vertex,minimum size=5pt] (x1) at (0.5,0.8) {} ; \node[] at (0.6,0.5) {$x_{1}$} ;
		\node[vertex,minimum size=5pt] (x2) at (1.5,0) {} ; \node[] at (1.8,0) {$x_{2}$} ;
		\node[vertex,minimum size=5pt] (x3) at (0.5,-0.8) {} ; \node[] at (0.5,-1.1) {$x_{3}$} ;
		
		\draw (a) -- (c) -- (b) -- (d) ;
		\draw (v) -- (x1) ; \draw (v) -- (x2) ; \draw (v) -- (x3) ;
		\draw (a) -- (x1) -- (x2) -- (x3) -- (b) ;
		
		\draw (-0.8,-1.5) rectangle (2,1.5) ;
		\node[] at (1.6,1.8) {$B$} ;

		\end{tikzpicture}
		\caption{An illustration for the proof of Claim \ref{claim:s<8}.}
		\label{figure3}
	\end{center}
\end{figure}
 
 Therefore, in order to finish the proof of Proposition \ref{prop:1/4}, it is enough to consider the cases $s=1,\dots,7$. The following simple inequality will help us to deal with these few cases.
 
 \begin{claim}\label{claim:inequality}
 	Let $0\leq a\leq b\leq c\leq d$. Then
 	$$ad+bc\leq ac+bd\leq ab+cd.$$
 \end{claim}

\begin{proof}
	The first inequality is equivalent to $(b-a)(d-c)\geq 0$, while the second inequality is equivalent to $(c-a)(d-b)\geq 0$.
\end{proof}

Without loss of generality, suppose that $\phi(1)\leq \phi(2)\leq\dots\leq \phi(s)$. Now let us consider the different possible values of $s$. For every value of $s$, we show that $\phi(Q)$ can be lower bounded by the sum of the squares of 4 numbers whose sum is equal to 1. But then $\phi(Q)\geq \frac{1}{4}$.

   \begin{description}
   	\item[$\mathbf{s\leq 4}$.] In this case, we simply have$$\phi(Q)\geq \sum_{i=1}^{s}\phi(i)^{2}\geq \frac{1}{s}\geq \frac{1}{4}.$$
   	\item[$\mathbf{s=5}$.] If $Q$ is not $K_{5}$-admissible, then $Q$ has at least two edges by Observation \ref{obs:basic}. But then
   	$$\phi(Q)\geq \sum_{i=1}^{5}\phi(i)^{2}+2\phi(1)\phi(2)=(\phi(1)+\phi(2))^{2}+\phi(3)^{2}+\phi(4)^{2}+\phi(5)^{2}\geq \frac{1}{4}.$$
   	\item[$\mathbf{s=6}$.] By Observation \ref{obs:basic}, there are 3 vertices $a,b,c\in \{2,3,4,5,6\}$ such that $ab,ac\in E(Q)$. But then $\phi(a)\phi(b)+\phi(a)\phi(c)\geq \phi(2)\phi(3)+\phi(2)\phi(4)$. Also, there are two disjoint edges $a_{1}b_{1},a_{2}b_{2}$ in $Q[[6]\setminus\{a\}]$. But then by Claim \ref{claim:inequality}, we have $\phi(a_{1})\phi(b_{1})+\phi(a_{2})\phi(b_{2})\geq \phi(1)\phi(4)+\phi(2)\phi(3)$. As $ab,ac,a_{1}b_{1},a_{2}b_{2}$ are four distinct edges of $Q$, we can write
   	\begin{align*}
   	\phi(Q)&\geq \sum_{i=1}^{6}\phi(i)^{2}+\phi(2)\phi(3)+\phi(2)\phi(4)+\phi(1)\phi(4)+\phi(2)\phi(3)\\
   	       &\geq (\phi(1)+\phi(4))^{2}+(\phi(2)+\phi(3))^{2}+\phi(5)^{2}+\phi(6)^{2}\\
   	       &\geq\frac{1}{4}.
    \end{align*}    
    \item[$\mathbf{s=7}$.] The analysis of this case is quite tedious, so we postpone it until the Appendix.
	\end{description}
\end{proof}

Now we are ready to conclude the proof of Lemma \ref{lemma:admissible2}, and therefore the proof of Lemma \ref{lemma:admissible1}.
\begin{proof}[Proof of Lemma \ref{lemma:admissible2}]
	The choice $k_{1}=\lceil \frac{1}{\epsilon_{3}}\rceil+1$ suffices. Indeed, we proved that if $(R,w)$ does not contain an $(H,0)$-admissible subgraph for any $H\in \mathcal{H}$, then $\phi(Q)\geq \frac{1}{4}$. But by (\ref{equ1}), we have 
	$$w(R)=\frac{k^{2}}{2}\left(\phi(Q)-\frac{1}{k}\right)>\left(\frac{1}{4}-\epsilon_{3}\right)\frac{k^{2}}{2},$$
	 finishing the proof.
\end{proof}

\subsection{Proof of Theorem \ref{thm:weakK5}}

In this section, we put everything together to finish the proof of Theorem \ref{thm:weakK5}.

\begin{proof}[Proof of Theorem \ref{thm:weakK5}.]
	Let $\mathcal{H}$ be the family of partial subdivisions of $K_{5}$ with at most $8$ vertices. As a reminder, $C$ is the constant given by Lemma \ref{lemma:separator}.
	
	Let $\epsilon_{2}=\epsilon$, $\epsilon_{1}=\min\{\frac{C}{2},\epsilon_{1}(\epsilon_{2})\}$ and $k_{0}=k_{0}(\epsilon_{2})$, where $\epsilon_{1}(\epsilon_{2})$ and $k_{0}(\epsilon_{2})$ are the constants given by Lemma \ref{lemma:admissible1}. Let $\beta=C$, and $\lambda=\frac{\beta}{8h}(\frac{\epsilon_{1}}{2})^{128}$. By the Regularity lemma, there exists $M=M(k_{0},\lambda)$ such that $G$ has a $\lambda$-regular partition $(V_{1},\dots,V_{k})$, where $k_{0}\leq k\leq M$. Let $(R,w)$ be the corresponding reduced graph. Then $|E(R)|\geq \binom{k}{2}-\lambda k^{2}> (1-\epsilon_{1})\frac{k^{2}}{2}.$ Also, $$\frac{n^{2}}{k^{2}}w(R)\leq |E(G)|\leq \left(\frac{1}{4}-\epsilon_{2}\right)\frac{n^{2}}{2},$$ so $w(R)\leq (\frac{1}{4}-\epsilon_{2})\frac{k^{2}}{2}$. But then by Lemma \ref{lemma:admissible1}, $(R,w)$ contains an $(H,\epsilon_{1})$-admissible subgraph for some $H\in \mathcal{H}$. Let $h=|V(H)|\leq 8$. 
	
	Let $\delta=\frac{1}{2M}(\frac{\epsilon_{1}}{2})^{128}$ and $\alpha=4\delta<\frac{1}{2M}(\frac{\epsilon_{1}}{2})^{16}$. Then the parameters $h,\alpha,\beta,\delta,\lambda,\epsilon_{1}$ satisfy the conditions of Lemma \ref{lemma:embedding}. Hence, if $G$ is $(\alpha,\beta)$-dense and $\delta$-full, then $G$ contains a weak-2-subdivision of $H$ as an induced subgraph. Note that a weak-subdivision of a partial subdivision of $K_{5}$ might not be a weak-subdivision of $K_{5}$, but it contains a weak-subdivision of $K_{5}$ as an induced subgraph. Thus, $G$ contains a weak-subdivision of $K_{5}$ as an induced subgraph. 
	
\end{proof}

\section{Concluding remarks}\label{sect:remarks}

It follows from the combination of Theorem \ref{thm:subdivision} and Theorem \ref{thm:weakK5} that if $G$ is a graph with $n$ vertices, at most $(\frac{1}{4}-\epsilon)\frac{n^{2}}{2}$ edges such that $G$ does not contain a weak-subdivision of $K_{5}$ as an induced subgraph, then $\overline{G}$ contains a linear sized bi-clique. It would be interesting to decide whether this statement can be generalized for graphs with no induced weak-subdivision of the complete graph $K_{t}$.

\begin{conjecture}\label{conj:Kt}
	For every $\epsilon>0$ and integer $t\geq 3$, there exists $\delta>0$ such that the following holds. If $G$ is a graph with $n$ vertices and at most $(\frac{1}{t-1}-\epsilon)\frac{n^{2}}{2}$ edges such that $G$ does not contain an induced weak-subdivision of $K_{t}$, then $\overline{G}$ contains a bi-clique of size at least $\delta n$. 
\end{conjecture}

Note that if this conjecture is true, it is sharp. Indeed, if $G$ is a graph whose vertex set can be partitioned into $t-1$ parts $V_{1},\dots, V_{t-1}$ such that $V_{i}$ spans a clique for $i=1,\dots,t-1$, then $G$ does not contain a weak-subdivision of $K_{t}$. But using standard probabilistic techniques, it is easy to construct such graphs $G$ with less than $(\frac{1}{t-1}+\epsilon)\frac{n^{2}}{2}$ edges such that the size of the largest bi-clique in $\overline{G}$ is $O(\frac{1}{\epsilon}\log n)$. Moreover, a positive answer to Conjecture \ref{conj:Kt} would have similar implications as Theorem \ref{thm:mainthm} for intersection graphs of curves defined on certain surfaces. We omit the details.

In order the prove Conjecture \ref{conj:Kt}, it would be enough to prove the following generalization of Proposition \ref{prop:1/4}, as this is the only part of our proof that does not apply to general families of graphs $\mathcal{H}$.

\begin{conjecture}\label{conj:newprop}
	Let $\mathcal{H}$ be the family of partial subdivisions of $K_{t}$. If $(Q,\phi)$ has no $H$-admissible subgraph for any $H\in\mathcal{H}$, then $\phi(Q)\geq \frac{1}{t-1}$. 
\end{conjecture}

We can prove the following slightly weaker version of Conjecture \ref{conj:Kt}, by proving a slightly weaker version of Conjecture \ref{conj:newprop}.

\begin{theorem}
	For every integer $t\geq 3$, there exists $\delta>0$ such that the following holds. If $G$ is a graph with $n$ vertices and at most $\frac{1}{2(t-1)}\cdot\frac{n^{2}}{2}$ edges such that $G$ does not contain an induced weak-subdivision of $K_{t}$, then $\overline{G}$ contains a bi-clique of size at least $\delta n$. 
\end{theorem}

\begin{proof}[Sketch proof]
  It is enough to show that there exists $\epsilon>0$ such that if $(Q,\phi)$ has no $K_{t}$-admissible subgraph, then $\phi(Q)\geq \frac{1}{2(t-1)}+\epsilon$. 
  
  Let $(Q,\phi)$ be the vertex weighted graph on vertex set $[s]$ with no $K_{t}$-admissible subgraph. Then $s< R(t)$, where $R(t)$ is the usual Ramsey number, that is, $R(t)$ is the smallest integer $N$ such that every graph on $N$ vertices contains either a clique or independent set of size $t$. Indeed, if $s\geq R(t)$, then $Q$ contains either a clique or independent set of size $t$, but both are $K_{t}$-admissible.
  
    Now we show that if  $\phi(Q)\geq \frac{1}{2s}+\frac{1}{2(t-1)}$, then $Q$ contains an independent set of size $t$, which is impossible. But then setting $\epsilon=\frac{1}{2R(t)}$ finishes the proof.
    
    By a similar Tur\'an type argument as in Proposition \ref{prop:partition}, we can show that among the vertex weighted graphs $(Q,\phi)$ with vertex set $[s]$, fixed weight function $\phi$, and no independent set of size $t$, the one that minimizes $\phi(Q)$ has the following form. The vertex set of $Q$ can be partitioned into $t-1$ parts $V_{1},\dots,V_{t-1}$ such that $V_{i}$ induces a clique in $Q$ for $i=1,\dots,t-1$. But then
  \begin{align*}
  \phi(Q)&=\sum_{i=1}^{s}\phi(i)^{2}+\sum_{i=1}^{t-1}\sum_{a,b\in V_{i},a<b}\phi(a)\phi(b)\\
  &=\frac{1}{2}\sum_{i=1}^{s}\phi(i)^{2}+\frac{1}{2}\sum_{i=1}^{t-1}\left(\sum_{a\in V_{i}}\phi(a)\right)^{2}\\
  &\geq \frac{1}{2s}+\frac{1}{2(t-1)}.
  \end{align*}
\end{proof}

\section{Acknowledgments}
I would like to thank J\'anos Pach for valuable discussions.

\section*{Appendix - Proof of Proposition \ref{prop:1/4}, $s=7$}

 Suppose that $s=7$. First, we show that if $Q$ contains either a cycle of length 6 or 7, we have $\phi(Q)\geq \frac{1}{4}$. Indeed, let $a_{1}<\dots<a_{r}$ be the vertices of this cycle, where $r\in \{6,7\}$, and let $\pi\in S_{r}$ be a permutation such that $a_{\pi(1)}a_{\pi(2)},a_{\pi(2)}a_{\pi(2)},\dots, a_{\pi(r)}a_{\pi(1)}$ are the edges of this cycle. Then
 $$\sum_{i=1}^{r}\phi(a_{\pi(i)})\phi(a_{\pi(i+1)})\geq \sum_{i=1}^{r}\phi(a_{i})\phi(a_{r+1-i})\geq 2\phi(1)\phi(6)+2\phi(2)\phi(5)+2\phi(3)\phi(4),$$
 where the first inequality is the consequence of the Rearrangement inequality \cite{HLP52}. Hence, 
\begin{align*}
\phi(Q)&\geq\sum_{i=1}^{7}\phi(i)^{2}+2\phi(1)\phi(6)+2\phi(2)\phi(5)+2\phi(3)\phi(4)\\
&\geq (\phi(1)+\phi(6))^{2}+(\phi(2)+\phi(5))^{2}+(\phi(3)+\phi(4))^{2}+\phi(7)^{2}\\
&\geq \frac{1}{4}.
\end{align*}

By Observation \ref{obs:deg5}, every vertex of $Q$ has degree at least $2$. Suppose that $Q$ has a vertex $a$ of degree 2, and let $b_{1},b_{2}$ be the vertices joined to $a$, and let $C=\{c_{1},c_{2},c_{3},c_{4}\}$ be the rest of the vertices. By Observation \ref{obs:K4}, $Q[C]$ is either a cycle of length 4 or path of length 3.

\begin{description}
	\item[Case 1.] $C$ is a cycle of length $4$. Without loss of generality, let $\phi(c_{1})\leq \phi(c_{2})\leq \phi(c_{3})\leq \phi(c_{4})$ and $\phi(b_{1})\leq \phi(b_{2})$. Then
	\begin{align*} 
	\phi(Q)&=\sum_{i=1}^{7}\phi(i)^{2}+\sum_{xy\in E(Q)}\phi(x)\phi(y)\\
	&\geq \sum_{i=1}^{7}\phi(i)^{2}+2\phi(c_{1})\phi(c_{4})+2\phi(c_{2})\phi(c_{3})+2\phi(a)\phi(b_{1})\\
	&=(\phi(c_{1})+\phi(c_{4}))^{2}+(\phi(c_{2})+\phi(c_{3}))^{2}+(\phi(a)+\phi(b_{1}))^{2}+\phi(b_{2})^{2}\\
	&\geq \frac{1}{4},
	\end{align*}
	where the first inequality holds by two applications of Claim \ref{claim:inequality}.
	
	\item[Case 2.] $C$ is a path of length $3$. Without loss of generality, let the edges of this cycle be $c_{1}c_{2},c_{2}c_{3},c_{3}c_{4}$. As $c_{1}$ and $c_{2}$ both have degree at least $2$, there is an edge from both $c_{1}$ and $c_{4}$ to $\{b_{1},b_{2}\}$. Without loss of generality, suppose that $c_{1}b_{1}$ is an edge. If $c_{4}b_{2}$ is an edge, then $Q$ contains a cycle of length $7$, namely $c_{1}c_{2}c_{3}c_{4}b_{2}ab_{1}$, so we are done. Hence, we may assume that $c_{4}b_{1}$ is an edge, and $c_{1}b_{2}$ and $c_{4}b_{2}$ are non-edges. But then both $b_{2}c_{2}$ and $b_{2}c_{3}$ are edges, because if $b_{2}c_{2}$ is a non-edge, say, then there are no edges between $\{a,b_{2}\}$ and $\{c_{1},c_{2},c_{4}\}$, contradicting Observation \ref{obs:K23}. But then $Q$ contains cycle of length $6$, namely $c_{1}b_{1}c_{4}c_{3}b_{2}c_{2}$, so we are done. See Figure \ref{figure4}.
\end{description}

	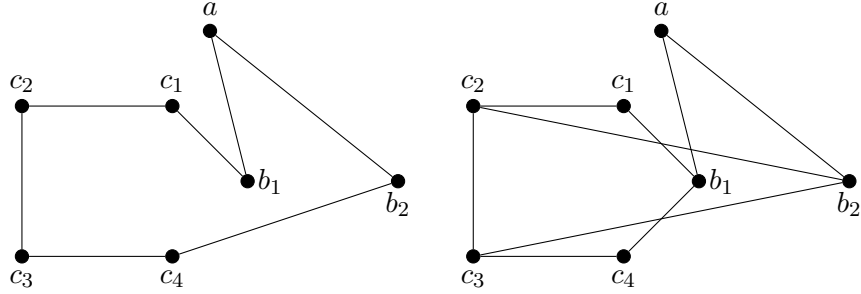
\begin{figure}[t]
	\begin{center}
		\begin{tikzpicture}[scale=1]
		
		\node[vertex,minimum size=5pt] (a) at (-6,2) {} ; \node[] at (-6,2.3) {$a$} ;
		\node[vertex,minimum size=5pt] (c1) at (-6.5,1) {} ; \node[] at (-6.5,1.3) {$c_{1}$} ;
		\node[vertex,minimum size=5pt] (c2) at (-8.5,1) {} ; \node[] at (-8.5,1.3) {$c_{2}$} ;
		\node[vertex,minimum size=5pt] (c4) at (-6.5,-1) {} ; \node[] at (-6.5,-1.3) {$c_{4}$} ;
		\node[vertex,minimum size=5pt] (c3) at (-8.5,-1) {} ; \node[] at (-8.5,-1.3) {$c_{3}$} ;
		
		\node[vertex,minimum size=5pt] (b1) at (-5.5,0) {} ; \node[] at (-5.2,0) {$b_{1}$} ;
		
		\node[vertex,minimum size=5pt] (b2) at (-3.5,0) {} ; \node[] at (-3.5,-0.3) {$b_{2}$} ;
		
		\draw (c1) -- (c2) -- (c3) -- (c4) ;
		\draw (a) -- (b1) ; \draw (a) -- (b2) ;
		\draw (c1) -- (b1) ; \draw (c4) -- (b2) ;
	%	\draw (b2) -- (c2) ; \draw (b2) -- (c3) ;
		
		\node[vertex,minimum size=5pt] (a) at (0,2) {} ; \node[] at (0,2.3) {$a$} ;
		\node[vertex,minimum size=5pt] (c1) at (-0.5,1) {} ; \node[] at (-0.5,1.3) {$c_{1}$} ;
		\node[vertex,minimum size=5pt] (c2) at (-2.5,1) {} ; \node[] at (-2.5,1.3) {$c_{2}$} ;
		\node[vertex,minimum size=5pt] (c4) at (-0.5,-1) {} ; \node[] at (-0.5,-1.3) {$c_{4}$} ;
		\node[vertex,minimum size=5pt] (c3) at (-2.5,-1) {} ; \node[] at (-2.5,-1.3) {$c_{3}$} ;
		
		\node[vertex,minimum size=5pt] (b1) at (0.5,0) {} ; \node[] at (0.8,0) {$b_{1}$} ;
	
		\node[vertex,minimum size=5pt] (b2) at (2.5,0) {} ; \node[] at (2.5,-0.3) {$b_{2}$} ;
		
		\draw (c1) -- (c2) -- (c3) -- (c4) ;
		\draw (a) -- (b1) ; \draw (a) -- (b2) ;
	    \draw (c1) -- (b1) ; \draw (c4) -- (b1) ;
	    \draw (b2) -- (c2) ; \draw (b2) -- (c3) ;
		
		\end{tikzpicture}
		\caption{An illustration for Case 2. Left is the subcase where $c_{4}b_{2}$ is an edge, right is the subcase where $c_{4}b_{2}$ is not an edge.}
		\label{figure4}
	\end{center}
\end{figure}

Therefore, we have $\phi(Q)\geq\frac{1}{4}$ if $Q$ contains a vertex with degree 2. Hence, we can suppose that every vertex has degree at least 3 in $Q$. As 7 is odd, $Q$ has at least one vertex $a$ of even degree, so the degree of $a$ is either 4 or 6. 

First, suppose that $a$ has degree $4$, and let $B=\{b_{1},b_{2},b_{3},b_{4}\}$  be the neighbors of $a$, and $c_{1},c_{2}$ be the rest of the vertices. By Observation \ref{obs:K4}, $Q[B]$ is either a cycle of length 4, or a path of length 3. In both cases $Q[B\cup\{a\}]$ contains a cycle $C$ of length $5$. If $c_{1}c_{2}$ is not an edge, then there are at least three edges between $c_{1}$ and $C$, so there are two consecutive vertices of $C$ joined to $c_{1}$. But then $C\cup \{c_{1}\}$ contains a cycle of length 6, so we are done. Therefore, we can suppose the $c_{1}c_{2}$ is an edge. Then there are at least 2 edges between $c_{i}$ and $C$ for $i=1,2$. In this case, we can find two disjoint edges between $\{c_{1},c_{2}\}$ and $C$, which also implies the existence of a cycle of length at least 6 in $C$.

The only remaining case is when the degree of $a$ is $6$. Let $b\in V(Q)\setminus \{a\}$, and let $A$ be the neighborhood of $b$ in $V(Q)\setminus \{a\}$. If $|A|\geq 3$, then we get a contradiction by 3. in Observation \ref{obs:K23}, and if $|A|\leq 2$, then $|V(Q)\setminus (A\cup\{a\})|\geq 3$, so we get a contradiction by 2. in Observation \ref{obs:K23}. 

\end{document}